\documentclass[12pt, twoside, fleqn, a4paper]{article}

\usepackage{latexsym}  
\usepackage{amscd}    
\usepackage{theorem}  
\usepackage{pifont}  
\usepackage{mathbbol} 
\usepackage{amsfonts}  
\usepackage{xspace}  
\usepackage{amssymb} 
\usepackage{fancyhdr} 
\usepackage{mathrsfs} 
\usepackage{amsmath} 
\usepackage{graphics} 
\usepackage{graphicx} 
\usepackage{euscript}
\usepackage{psfrag}
\usepackage{empheq} 
\usepackage{mathabx} 
\usepackage[parfill]{parskip}     
\usepackage[colorlinks=true, allcolors=blue]{hyperref}
\usepackage{cancel} 
\usepackage{bigints} 

\title{Uniqueness of the equilibrium state in the dynamics of holomorphic correspondences}
\author{Bharath Krishna Seshadri\footnote{Indian Institute of Science Education and Research Thiruvananthapuram (IISER-TVM). \\ ORCID: 0009 - 0004 - 0541 - 8040\ \ email: \texttt{bharathmaths21@iisertvm.ac.in}}\ \ \ and\ \ \ Shrihari Sridharan\footnote{Indian Institute of Science Education and Research Thiruvananthapuram (IISER-TVM). \\ ORCID: 0000 - 0003 - 2434 - 4767\ \ email: \texttt{shrihari@iisertvm.ac.in}\ (Corresponding Author) \\ The second named author thanks the support provided by NBHM through a research grant \\ No. 02011/35/2025/NBHM(R.P)/R\&D II/9832}}

\DeclareFontFamily{OT1}{pzc}{}
\DeclareFontShape{OT1}{pzc}{m}{it}%
              {<-> s * [0.900] pzcmi7t}{}
\DeclareMathAlphabet{\mathpzc}{OT1}{pzc}%
                                 {m}{it}

{\theorembodyfont{\slshape} \newtheorem{theorem}{Theorem}[section]}
{\theorembodyfont{\slshape} \newtheorem{definition}[theorem]{Definition}}
{\theorembodyfont{\slshape} \newtheorem{lemma}[theorem]{Lemma}}
{\theorembodyfont{\slshape} \newtheorem{proposition}[theorem]{Proposition}}
{\theorembodyfont{\slshape} }
{\theorembodyfont{\slshape} } 
{\theorembodyfont{\slshape} } 
{\theorembodyfont{\slshape} \newtheorem{remark}[theorem]{Remark}}
{\theorembodyfont{\slshape} }

\numberwithin{equation}{section}

\newenvironment{proof}{\paragraph{Proof:}}{\hfill$\bullet$}

\begin{document}

\maketitle

\begin{abstract} 
\noindent 
This paper concerns the study of the existence of a unique equilibrium state for a H\"{o}lder continuous function under the dynamics of a holomorphic correspondence defined on the Riemann sphere. We mainly work with the correspondence restricted on the support of the Dinh-Sibony measure and identify topologically interesting correspondences, namely distance expanding ones. Further, we consider H\"{o}lder continuous potentials defined on the support of the Dinh-Sibony measure, for which we prove the uniqueness of equilibrium state. Along the way, we also prove some interesting topological results related to holomorphic correspondences. Finally, we establish a result connecting the Ruelle operator for holomorphic correspondences and the unique equilibrium state under a suitable hypothesis. The concluding part of the paper is devoted to some discussion related to the hypothesis involved and providing some examples. 
\end{abstract}

\begin{tabular}{l l} 
{\bf Keywords} & Unique equilibrium state \\ 
& Distance expanding correspondences \\ 
& Entropy and pressure \\ 
& Ruelle operator \\
& \\
& \\
{\bf MSC Subject} & 37A35, 37A60, 37F80 \\ 
{\bf Classifications} & 
\end{tabular} 
\bigskip 

\newpage 

\section{Introduction} 

The concept of thermodynamic formalism in dynamical systems is mainly associated with the study of quantities called entropy and pressure (associated with a real-valued potential) and the transfer operator (pertaining to a fixed potential) under the influence of the dynamics of the system under consideration. This theory is well-developed in the case where the dynamical system is constructed by iterating a continuous map, say $T$ on a compact metric space, say $X$. In this paper, we are interested in working with a dynamical system obtained by iterating a holomorphic correspondence defined on the Riemann sphere $\widehat{\mathbb{C}} = \mathbb{C} \cup \{ \infty \}$, which naturally has a complex analytic flavour. This dynamical system is a generalisation of the case of rational maps defined on $\widehat{\mathbb{C}}$. 

Let $T : X \longrightarrow X$ be a continuous map defined on a compact metric space and let $\phi \in \mathcal{C} (X, \mathbb{R})$, the Banach space of real-valued continuous functions defined on $X$. Let $M(T)$ denote the space of all Borel probability measures supported on $X$ that are $T$-invariant. The variational principle, proved in \cite{wal:1975} by Walters, states that $\displaystyle{{\rm Pr} (T, \phi) = \sup\limits_{\mu\, \in\, M(T)} \left\{ h_{\mu} (T) + \int \phi \mathrm{d}\mu \right\}}$, where ${\rm Pr} (T, \phi)$ denotes the pressure of $\phi$ with respect to the dynamics of $T$ and $h_{\mu} (T)$ denotes the entropy of $T$ with respect to the measure $\mu$. An interesting question about this variational principle is about the existence of a unique measure, say $\mu_{\phi}$ that satisfies $\displaystyle{{\rm Pr} (T, \phi) = h_{\mu_{\phi}} (T) + \int \phi \mathrm{d}\mu_{\phi}}$. For the particular case, when $\phi \equiv 0$, the constant zero function, it is known that ${\rm Pr} (T, \phi)$ equals the topological entropy of the dynamical system. In this case, Walters studied maps $T$ for which a unique measure of maximal entropy exists, in the work \cite{wal:1974}. In addition, he also studied the invariant measures and equilibrium states of certain expanding maps in \cite{wal:1978}. A good compilation of these ideas across various dynamical systems can be found in books such as \cite{rue:1978, pp:1990, zins:2000, pu:2010}.

With regards to the dynamics of rational maps defined on the Riemann sphere, the works \cite{pry:1990, du:1991} investigated the problem of the existence of a unique equilibrium state and studied its ergodic theoretic properties. These works considered the dynamics of a rational map $R$ restricted on its Julia set $J_{R}$. Unlike the works where the dynamics was defined by a continuous map on some compact space (typically the shift map and the interval map), the work \cite{du:1991} by Denker and Urba\'{n}ski did not impose any expanding condition on the rational map. Rather, they worked with the H\"{o}lder continuous potentials $\phi \in \mathcal{C} (J_{R}, \mathbb{R})$ that satisfy the condition ${\rm Pr} (R, \phi) > \sup \phi$ and proved results on the uniqueness of the equilibrium state corresponding to $\phi$, using the theory of conformal measures. In \cite{pry:1990}, Przytycki explored the same question using different techniques and proved various results. 

Motivated by these works, we investigate the existence of an equilibrium state and the uniqueness of such a state, for appropriate potentials under the dynamics of a holomorphic correspondence defined on $\widehat{\mathbb{C}}$. We focus our attention on distance expanding holomorphic correspondences, which are defined analogous to the case of the distance expanding maps. Once again, the H\"{o}lder continuous potentials turn out to be interesting, in our setting too. One of the crucial differences in the study of correspondences is that they are generally set-valued, and thus, typical points have multiple forward images, as well. Nevertheless, recent works \cite{ss:2025} and \cite{ss:2026} have enabled us to consider this problem in the setting of holomorphic correspondences. The work \cite{ss:2026} proved a variational principle for holomorphic correspondences after introducing the notions of measure theoretic entropy and pressure in this setting. This naturally begs the question concerning the idea of the uniqueness of the equilibrium state. Another recent work namely \cite{llz:2023} explored thermodynamic formalism for correspondences (set-valued maps on a compact space), although their notions of entropy and pressure are different from our construction. Further, our main focus is to study holomorphic correspondences and as a result, we also prove some results exclusive to this setting (see Section \ref{sec:four} in particular), which one obtains by appealing to the underlying analytic structure. These interesting results also play a role in the proof of our main theorems. We are now ready to state the main results of the paper. 

\begin{theorem}
\label{thm:uniqueness} 
Let $\digamma$ be a distance expanding holomorphic correspondence defined on $\widehat{\mathbb{C}}$. Let $\Omega_{{\rm DS}}$ denote the support of the Dinh-Sibony measure of $\digamma$. Then, any H\"{o}lder continuous function $f \in \mathcal{C} \left( \Omega_{{\rm DS}}, \mathbb{R} \right)$ has a unique equilibrium state. 
\end{theorem}

The set $\Omega_{{\rm DS}}$, introduced in Theorem \ref{thm:uniqueness}, was first studied in \cite{ds:2006} and is observed to be the analogue of the Julia set from the study of rational maps and hence it is the natural domain to consider, in our setting. In Sections \ref{sec:three} and \ref{sec:four}, we define and provide more details on the set $\Omega_{{\rm DS}}$ and its topological properties that makes it interesting for one to study the restriction of $\digamma$ on the set $\Omega_{{\rm DS}}$.

An important object in thermodynamic formalism is the Ruelle operator. In the case of dynamics of rational maps, the unique equilibrium state has some nice connections with the Ruelle operator, as explored by Przytycki in \cite{pry:1990} and Denker and Urba\'{n}ski in \cite{du:1991}. The Ruelle operator in the setting of holomorphic correspondences was studied in detail in \cite{ss:2025}. We now state a result that connects the unique equilibrium state obtained from Theorem \ref{thm:uniqueness} with the Ruelle operator, in this setting. For a set $X$ we use the notation $\mathbb{1}_{X}$ to denote the real valued function that takes the constant value $1$ at all points of $X$. The space $\mathscr{S}^{\Gamma}$ mentioned in the following theorem will be introduced in Section \ref{sec:three}. Further, ${\rm Pr} (\Gamma, f)$ denotes the pressure of potential $f$ with respect to the dynamics of $\Gamma$ and $\mathcal{L}_{f}$ denotes the Ruelle operator associated to $f$.

\begin{theorem}
\label{thm:ruelleuniqueconnnect}
Let $\digamma$ be a distance expanding holomorphic correspondence on $\widehat{\mathbb{C}}$. Let $\Omega_{{\rm DS}}$ denote the support of the Dinh-Sibony measure of $\digamma$. Consider a H\"{o}lder continuous function $f \in \mathcal{C} (\Omega_{{\rm DS}}, \mathbb{R})$ satisfying $\mathcal{L}_{f} (\mathbb{1}_{\Omega_{{\rm DS}}}) = e^{{\rm Pr} (\Gamma, f)} (\mathbb{1}_{\Omega_{{\rm DS}}})$. Then, there exists a measure $\nu_{f} \in \mathscr{S}^{\Gamma}$ such that $\mathcal{L}_{f}^{*} \nu_{f} = e^{{\rm Pr} (\Gamma, f)} \nu_{f}$, where $\mathcal{L}_{f}^{*}$ denotes the dual to the Ruelle operator $\mathcal{L}_{f}$. Further, this measure $\nu_{f}$ also turns out to be the unique equilibrium state for the potential $f$.
\end{theorem}

The paper is organised as follows. In Section \ref{sec:two}, we study the preliminaries on holomorphic correspondences that we require in our work. This section introduces the spaces that we work with and describes the ways in which we study the dynamics in this setting. Section \ref{sec:three} elaborates on the latest results on thermodynamic formalism in the setting of holomorphic correspondences. Here, we also define the measure theoretic entropy and pressure and state some known properties. In Section \ref{sec:four}, we focus on the topological properties of a holomorphic correspondence. These results are crucial to prove the main theorems of this work. The main result of this section is Theorem \ref{thm:open}, which can be considered as an analogue of the open mapping theorem for analytic functions from complex analysis. Section \ref{sec:five} is devoted to proving our two main theorems. Apart from showing the existence of a unique equilibrium state, we also touch upon its connection with the Ruelle operator. In the concluding section, namely Section \ref{sec:six}, we engage the audience with a discussion on the hypothesis in Theorem \ref{thm:ruelleuniqueconnnect}. We also connect our results with some examples of well-known holomorphic correspondences from other works. 

\section{Preliminaries on holomorphic correspondences}
\label{sec:two}

In this section, we introduce the dynamical system under consideration in our study. We are interested in the Riemann sphere $\widehat{\mathbb{C}}$, which is a metric space when accorded with the spherical metric, which we denote by $\rho$. The object of our interest is called holomorphic correspondence, which we shall define below. In the following definition, for $i = 1, 2$ we denote by $\pi_{i}$ the projection of points in $\widehat{\mathbb{C}} \times \widehat{\mathbb{C}}$ onto their $i^{\text{th}}$ coordinate. 

\begin{definition} 
\label{defn:holocorr} 
A holomorphic correspondence on $\widehat{\mathbb{C}}$ (denoted by $\Gamma$) is a formal sum of  (not necessarily distinct) irreducible complex subvarieties of $\widehat{\mathbb{C}} \times \widehat{\mathbb{C}}$ of dimension $1$, given by $\Gamma =  \sideset{}{'} \sum\limits_{j\, =\, 1}^{N} \Gamma_{j}$ that satisfies, 
\begin{enumerate}
\item For each $1 \le j \le N$, the maps $\left. \pi_{1} \right|_{\Gamma_{j}}$ and $\left. \pi_{2} \right|_{\Gamma_{j}}$ are surjective and,
\item The sets $\bigcup\limits_{j\, =\, 1}^{N} \pi_{1}^{-1} \{z_{0}\} \cap \Gamma_{j}$ and $\bigcup\limits_{j\,=\, 1}^{N} \pi_{2}^{-1} \{z_{0}\} \cap \Gamma_{j}$, where the points are repeated according to their multiplicities, are finite for every $z_{0} \in \widehat{\mathbb{C}}$.
\end{enumerate}
\end{definition} 

Note that the notation of the summation namely, $\sideset{}{'} \sum$ is used to highlight the fact that the varieties appearing in the sum are allowed to repeat, if necessary, according to the multiplicity of the variety. We call the set $\bigcup\limits_{j\, =\, 1}^{N} \pi_{2} \left( \pi_{1}^{-1} \{z_{0}\} \cap \Gamma_{j} \right)$ to be the \emph{set of images} of the point $z_{0}$ under the correspondence whereas the \emph{set of pre-images} of the point $z_{0}$ is the set $\bigcup\limits_{j\,=\, 1}^{N} \pi_{1} \left( \pi_{2}^{-1} \{z_{0}\} \cap \Gamma_{j} \right)$. The images and pre-images are repeated according to their multiplicities. Thus, the correspondence $\Gamma$ induces a set valued map on $\widehat{\mathbb{C}}$, which we shall denote by $\digamma$. We shall use the notations $\Gamma$ and $\digamma$ interchangeably to denote the holomorphic correspondence under consideration, as per the demands of the situation. We define the image of a subset $A \subset \widehat{\mathbb{C}}$ to be $\digamma(A) := \bigcup\limits_{j\, =\, 1}^{N} \pi_{2} \left( \pi_{1}^{-1} (A) \cap \Gamma_{j} \right)$ and the pre-image of $A \subset \widehat{\mathbb{C}}$ to be $\digamma^{\dagger} (A) := \bigcup\limits_{j\,=\, 1}^{N} \pi_{1} \left( \pi_{2}^{-1} (A) \cap \Gamma_{j} \right)$. The cardinality of $\digamma(z_{0})$ for a generic $z_{0}$ is called the \emph{forward degree} of $\digamma$ (denoted by $d_{{\rm fwd}}$) and that of $\digamma^{\dagger} (z_{0})$ for a generic $z_{0}$ is called the topological degree of $\digamma$ (denoted by $d_{{\rm top}}$). The genericity of the points mentioned above is explained by a remark from \cite{lon:2024}, which also concludes that $d_{{\rm top}} = \sum\limits_{j\, =\, 1}^{N} \delta_{j}$, where $\delta_{j}$ is the cardinality of the set $\left\{ z \in \widehat{\mathbb{C}} : (z, w) \in \Gamma_{j} \right\}$, for a generic $w \in \widehat{\mathbb{C}}$. This is one way to do dynamics with a holomorphic correspondence. 

We now introduce another perspective of the same, by introducing the space of orbits associated to points in $\widehat{\mathbb{C}}$ {\it via} the correspondence $\Gamma$. We define the space of all forward orbits of length $n \in \mathbb{Z}_{+}$ under the iteration of $\Gamma$ by 
\begin{eqnarray*}
\mathscr{P}_{n}^{\Gamma} \left( \widehat{\mathbb{C}} \right) & = & \Bigg\{ \left( z_{0}, z_{j_{1}}^{(1)}, \cdots, z_{j_{n}}^{(n)};\; \alpha_{1}, \cdots, \alpha_{n} \right)\ :\ \left( z_{j_{r - 1}}^{(r - 1)}, z_{j_{r}}^{(r)} \right) \in \Gamma_{\alpha_{r}}\ \text{where}\ z_{j_{0}}^{(0)} = z_{0} \in \widehat{\mathbb{C}}, \\ 
& & \hspace{+3cm} \alpha_{r} \in \left\{ 1, 2, \cdots, N \right\}\ \text{and}\ 1 \le j_{r} \le \lambda_{\alpha_{r}} \left( z_{j_{r - 1}}^{(r - 1)} \right)\ \text{for}\ 1 \le r \le n \Bigg\},
\end{eqnarray*} 
where $\lambda_{\alpha_{r}} \left( z_{j_{r - 1}}^{(r - 1)} \right)$ denotes the cardinality of the set $\digamma \left (z_{j_{r - 1}}^{(r - 1)} \right) \cap \Gamma_{\alpha_{r}}$. We denote a point in $\mathscr{P}_{n}^{\Gamma} \left( \widehat{\mathbb{C}} \right)$ by $\mathfrak{X}^{+}_{n} \left( z_{0}; \boldsymbol{\alpha} \right)_{\boldsymbol{j}}$ where $z_{0} \in \widehat{\mathbb{C}}$ is the starting point of the orbit, $\boldsymbol{\alpha} = \left( \alpha_{1}, \alpha_{2}, \cdots ,\alpha_{n} \right)$ is the sequence of varieties through which the orbit progresses and $\boldsymbol{j} = \left( j_{1}, j_{2}, \cdots, j_{n} \right)$ distinguishes the images of the same point that are found in the same subvariety. The space $\mathscr{P}_{n}^{\Gamma} \left( \widehat{\mathbb{C}} \right)$ is compact as a subspace of $\left( \widehat{\mathbb{C}} \right)^{n + 1} \times \{1, 2, \cdots , N\}^{n}$ where the latter space is given the appropriate product topology by considering $\widehat{\mathbb{C}}$ to have the spherical metric topology and $\{1, 2, \cdots, N\}$ to have the discrete topology. We define the projections $\Pi_{(r, n)} : \mathscr{P}_{n}^{\Gamma} \left( \widehat{\mathbb{C}} \right) \longrightarrow \widehat{\mathbb{C}}$ given by $\Pi_{(r, n)}  \left( \mathfrak{X}^{+}_{n} \left( z_{0}; \boldsymbol{\alpha} \right)_{\boldsymbol{j}} \right) = z_{j_{r}}^{(r)}$ for $0 \le r \le n$ and ${\rm Proj}_{(r, n)} : \mathscr{P}_{n}^{\Gamma} \left( \widehat{\mathbb{C}} \right) \longrightarrow \{ 1, 2, \cdots, N \}$ given by ${\rm Proj}_{(r, n)} \left(\mathfrak{X}^{+}_{n} \left( z_{0}; \boldsymbol{\alpha} \right)_{\boldsymbol{j}}\right) = \alpha_{r}$, for $1 \le j \le n$.

The space of infinitely long forward orbits under the correspondence starting at a point $z_{0} \in \widehat{\mathbb{C}}$ is defined as 
\begin{eqnarray*}
\mathscr{P}^{\Gamma} \left( z_{0} \right) & := & \Bigg\{ \left( z_{0}, z_{j_{1}}^{(1)}, z_{j_{2}}^{(2)}, \cdots;\; \alpha_{1}, \alpha_{2}, \cdots \right)\ :\ \left( z_{j_{r - 1}}^{(r - 1)}, z_{j_{r}}^{(r)} \right) \in \Gamma_{\alpha_{r}}\ \text{where}\ z_{j_{0}}^{(0)} = z_{0}, \\ 
& & \hspace{+3cm} \alpha_{r} \in \left\{ 1, 2, \cdots, N \right\}\ \text{and}\ 1 \le j_{r} \le \lambda_{\alpha_{r}} \left( z_{j_{r - 1}}^{(r - 1)} \right)\ \text{for}\ r \in \mathbb{Z}_{+} \Bigg\}.
\end{eqnarray*}
We further define the space of infinitely long forward orbits under the correspondence, with starting points in $\widehat{\mathbb{C}}$ to be $\mathscr{P}^{\Gamma} \left( \widehat{\mathbb{C}} \right) := \bigcup\limits_{z_{0}\, \in\, \widehat{\mathbb{C}}} \mathscr{P}^{\Gamma} \left( z_{0} \right)$. We denote a point in $\mathscr{P}^{\Gamma} \left( \widehat{\mathbb{C}} \right)$ by $\mathfrak{X}^{+} \left( z_{0}; \boldsymbol{\alpha} \right)_{\boldsymbol{j}}$ where $z_{0} \in \widehat{\mathbb{C}}$ is the starting point of the orbit, $\boldsymbol{\alpha} = \left( \alpha_{1}, \alpha_{2}, \cdots \right)$ is the sequence of varieties through which the orbit progresses and $\boldsymbol{j} = \left( j_{1}, j_{2}, \cdots \right)$ distinguishes the images of the same point that are found in the same subvariety. The projections $\Pi_{r}$ and ${\rm Proj}_{r}$ on the space $\mathscr{P}^{\Gamma} \left( \widehat{\mathbb{C}} \right)$ are defined analogous to the corresponding projections on $\mathscr{P}_{n}^{\Gamma} \left( \widehat{\mathbb{C}} \right)$ and are given by $\Pi_{r}  \left( \mathfrak{X}^{+} \left( z_{0}; \boldsymbol{\alpha} \right)_{\boldsymbol{j}} \right) = z_{j_{r}}^{(r)}$ for $r \in \mathbb{Z}_{+} \cup \{ 0 \}$ and ${\rm Proj}_{r} \left(\mathfrak{X}^{+} \left( z_{0}; \boldsymbol{\alpha} \right)_{\boldsymbol{j}}\right) = \alpha_{r}$ for $r \in \mathbb{Z}_{+}$. One can observe that $\mathscr{P}^{\Gamma} \left( \widehat{\mathbb{C}} \right)$ is a metric space when equipped with the metric 
\[ d \left( \mathfrak{X}^{+} \left( z_{0}; \boldsymbol{\alpha} \right)_{\boldsymbol{j}},\; \mathfrak{X}^{+} \left( w_{0}; \boldsymbol{\beta} \right)_{\boldsymbol{k}} \right)\ =\ \max \left\{ \sum\limits_{n\, \ge\, 0} \frac{1}{2^{n}} \frac{\rho \left( z_{j_{n}}^{(n)},\, w_{k_{n}}^{(n)} \right)}{\left( 1 + \rho \left( z_{j_{n}}^{(n)},\, w_{k_{n}}^{(n)} \right) \right)}\ ,\ \frac{1}{2^{\min \left\{ l\; \mid\; \alpha_{l} \ne \beta_{l} \right\}}} \right\}. \] 
This metric induces the topology on $\mathscr{P}^{\Gamma} \left( \widehat{\mathbb{C}} \right)$ that is obtained as the subspace topology of the product topology on $\left( \widehat{\mathbb{C}} \right)^{\mathbb{Z}_{+}} \times \big\{ 1, 2, \cdots , N \big\}^{\mathbb{Z}_{+}}$. Under this metric, the space $\mathscr{P}^{\Gamma} \left( \widehat{\mathbb{C}} \right)$ is compact.

Define a shift map $\sigma^{\Gamma} : \mathscr{P}^{\Gamma} \left( \widehat{\mathbb{C}} \right) \longrightarrow \mathscr{P}^{\Gamma} \left( \widehat{\mathbb{C}} \right)$ as 
\[ \sigma^{\Gamma} \left( \mathfrak{X}^{+} \left( z_{0}; \boldsymbol{\alpha} \right)_{\boldsymbol{j}} \right)\ \ =\ \ \left( z_{j_{1}}^{(1)}, z_{j_{2}}^{(2)}, \cdots;\; \alpha_{2}, \cdots \right)\ \ =\ \ \mathfrak{X}^{+} \left( z_{j_{1}}^{(1)}; \sigma \boldsymbol{\alpha} \right)_{\sigma \boldsymbol{j}}, \] 
where by prefixing $\sigma$ to an infinite lettered word $\boldsymbol{\alpha}$ or $\boldsymbol{j}$, we mean the infinite lettered word that we obtain from $\boldsymbol{\alpha}$ or $\boldsymbol{j}$ as appropriate, by dropping its first letter. Also, we use the notation $\sigma^{m} \boldsymbol{\alpha}$ and $\sigma^{m} \boldsymbol{j}$ to denote $\sigma \left( \sigma^{m - 1} \boldsymbol{\alpha} \right)$ and $\sigma \left( \sigma^{m - 1} \boldsymbol{j} \right)$ respectively, for any $m > 1$. $\sigma^{\Gamma}$ is a continuous map on $\mathscr{P}^{\Gamma} \left( \widehat{\mathbb{C}} \right)$ and serves as a tool to study the dynamics of the correspondence $\Gamma$.

With this background, we now introduce the space that we mainly focus on in the upcoming sections. Called the \emph{support of the Dinh-Sibony measure}, this space has its origins from a theorem due to Dinh and Sibony from \cite{ds:2006} and is the analogue of the Julia set from the dynamics of rational maps. Before proceeding further, we first state a result from \cite{ds:2006} which has been rephrased to fit to our context. 

\begin{theorem} \cite{ds:2006}
\label{thm:ds} 
Let $\digamma$ be a holomorphic correspondence on $\widehat{\mathbb{C}}$ with $d_{{\rm top}} > d_{{\rm fwd}}$. Then, there is an exceptional set $\mathcal{E} \subsetneq \widehat{\mathbb{C}}$ such that 
\[ \frac{1}{d_{{\rm top}}^{n}} \left( \digamma^{n} \right)^{*} \delta_{z_{0}}\ \ \rightharpoonup\ \ \mu_{\digamma},\ \ \ \ \text{in the weak* topology, for all}\ z_{0} \in \widehat{\mathbb{C}} \setminus \mathcal{E}. \] 
Here, $\left( \digamma^{n} \right)^{*} \delta_{z_{0}}$ denotes the pullback of the Dirac delta measure $\delta_{z_{0}}$ under the holomorphic correspondence $\digamma^{n}$. Moreover, $\mu_{\digamma}$ satisfies $\dfrac{1}{d_{{\rm top}}} \digamma^{*} \mu_{\digamma} = \mu_{\digamma}$.
\end{theorem} 

The authors, in \cite{bs:2016}, proved an equidistribution result for holomorphic correspondences satisfying $d_{{\rm top}} \le d_{{\rm fwd}}$, by considering a different hypothesis. 

\begin{theorem} \cite{bs:2016}
\label{thm:bsequi}
Let $\digamma$ be a holomorphic correspondence on $\widehat{\mathbb{C}}$ with $d_{{\rm top}} \le d_{{\rm fwd}}$. Assume that there exists a strong repeller, denoted by $\mathcal{R}$ that is disjoint from the set of critical values of $\digamma$. Then, there exists an open set $U \left( \digamma, \mathcal{R} \right) \supset \mathcal{R}$ such that  
\[ \frac{1}{d_{{\rm top}}^{n}} \left( \digamma^{n} \right)^{*} \delta_{z_{0}}\ \ \rightharpoonup\ \ \mu_{\digamma},\ \ \ \ \text{in the weak* topology, for all}\ z_{0} \in U \left( \digamma, \mathcal{R} \right). \] 
Moreover, $\mu_{\digamma}$ satisfies $\dfrac{1}{d_{{\rm top}}} \digamma^{*} \mu_{\digamma} = \mu_{\digamma}$. 
\end{theorem} 

Henceforth, we work with holomorphic correspondences $\digamma$ for which the measure $\mu_{\digamma}$, called the \emph{Dinh-Sibony measure}, as in the Theorem \ref{thm:ds} exists. We denote by $\Omega_{{\rm DS}} \subseteq \widehat{\mathbb{C}}$, the support of the Dinh-Sibony measure of the correspondence $\digamma$. $\Omega_{{\rm DS}}$ is a compact metric space, when the metric $\rho$ is restricted on $\Omega_{{\rm DS}}$. It is known that (refer \cite{lon:2022}) $\Omega_{{\rm DS}}$ is backward invariant with respect to $\digamma$ but is not forward invariant, in general. Our goal is to study the dynamics of $\digamma$ on the space $\Omega_{{\rm DS}}$ and obtain results analogous to the case of dynamics of rational maps on the Julia set with regards to thermodynamic formalism. However, the absence of forward invariance of $\Omega_{{\rm DS}}$ with respect to $\digamma$ poses a significant challenge, which is overcome, thanks to the following result proved by Londhe. 

\begin{theorem} \cite{lon:2022} 
Let $\digamma$ be a holomorphic correspondence on $\widehat{\mathbb{C}}$ with $d_{{\rm top}} > d_{{\rm fwd}}$ and $\mu_{\digamma}$ be its Dinh-Sibony measure, whose support is denoted by $\Omega_{{\rm DS}}$. If $z_{0} \in \Omega_{{\rm DS}}$, then for all $n \in \mathbb{Z}_{+}$, we have $\digamma^{n} (z_{0}) \cap \Omega_{{\rm DS}} \ne \emptyset$. 
\end{theorem} 

Motivated by the above theorem, it is possible to consider the dynamics of a holomorphic correspondence $\Gamma$ restricted on $\Omega_{{\rm DS}}$, {\it i.e.}, we consider only those points in $\Gamma$ that are also in $\Omega_{{\rm DS}} \times \Omega_{{\rm DS}}$. For convenience, we abuse notations and continue to denote the correspondence on $\Omega_{{\rm DS}}$ as $\Gamma$ and its corresponding set valued map as $\digamma$. The situation shall make it clear whether we work with the domain $\widehat{\mathbb{C}}$ or $\Omega_{{\rm DS}}$. Further, we define the orbit space for the correspondence $\Gamma$ on $\Omega_{{\rm DS}}$ to be $\mathscr{P}^{\Gamma}_{{\rm inv}} \left( \Omega_{{\rm DS}} \right) := \mathscr{P}^{\Gamma} \left( \widehat{\mathbb{C}} \right) \cap \left( \Omega_{{\rm DS}}^{\mathbb{Z}_{+}} \times \big\{ 1, 2, \cdots, N \big\}^{\mathbb{Z}_{+}} \right)$. Restricting the metric $d$ defined earlier to the space $\mathscr{P}^{\Gamma}_{{\rm inv}} \left( \Omega_{{\rm DS}} \right)$, we find that it is a compact metric space that is completely invariant under the shift map $\sigma^{\Gamma}$. One can similarly define the $n$ long permissible forward paths with respect to $\Gamma$ that entirely lie inside $\Omega_{{\rm DS}}$ to be the set $\left( \mathscr{P}_{n}^{\Gamma} \right)_{{\rm inv}} \left( \Omega_{{\rm DS}} \right) := \mathscr{P}_{n}^{\Gamma} \left( \widehat{\mathbb{C}} \right) \cap \left( \Omega_{{\rm DS}}^{n + 1} \times \big\{ 1, 2, \cdots, N \big\}^{n} \right)$.

\begin{remark}
It must be noted that equidistribution results along the lines of Theorem \ref{thm:ds} and Theorem \ref{thm:bsequi} have been proved for various families of correspondences, by now. Some works in this regard include \cite{dkw:2020, hem:2024, vd:2024}. It is possible to consider the set $\Omega_{{\rm DS}}$ and conduct investigations in those settings as well.
\end{remark}

\section{Thermodynamic formalism} 
\label{sec:three}

In this section, we elaborate on the results in the framework of thermodynamic formalism for holomorphic correspondences that we require in our work. The works \cite{ss:2025, ss:2026} form the basis of the concepts mentioned in this section. We work with the notions of entropy and pressure, related to holomorphic correspondences, which were studied in these two works. This is not the only way to study thermodynamic formalism in the setting of correspondences, as one can also refer to the methods in \cite{llz:2023}. We also bring to the readers' attention a few other works in the same spirit. The notion of topological entropy has been studied in the case of holomorphic correspondences in \cite{ds:2008} and \cite{bs:2021}. Another notion of entropy that is studied in the case of set valued maps is metric entropy, which was studied in detail in \cite{vs:2022}, a work that also proves a half variational principle. The recent work \cite{vd:2026} makes use of this metric entropy to study measures of maximal entropy for certain families of holomorphic correspondences.

For a compact metric space $X$, we denote the space of all Borel probability measures supported on $X$ by $\mathcal{M} (X)$. This space is compact in the weak* topology. Recall the projection map $\Pi_{0} : \mathscr{P}^{\Gamma}_{{\rm inv}} \left( \Omega_{{\rm DS}} \right) \to \Omega_{{\rm DS}}$, defined in Section \ref{sec:two}. This map induces a push forward between the corresponding space of measures given by $\left( \Pi_{0} \right)_{*} : \mathcal{M} \left( \mathscr{P}^{\Gamma}_{{\rm inv}} \left( \Omega_{{\rm DS}} \right) \right) \longrightarrow \mathcal{M} \left( \Omega_{{\rm DS}} \right)$ defined by $\left( \left( \Pi_{0} \right)_{*} \mu \right) (A) = \mu \left( \Pi_{0}^{-1} (A) \right)$ for any $\mu \in \mathcal{M} \left( \mathscr{P}^{\Gamma}_{{\rm inv}} \left( \Omega_{{\rm DS}} \right) \right)$ and for any Borel subset $A$ of $\Omega_{{\rm DS}}$. We work with the space $\mathscr{S}^{\Gamma} \subseteq \mathcal{M} \left( \Omega_{{\rm DS}} \right)$ defined as $\mathscr{S}^{\Gamma} := \left( \Pi_{0} \right)_{*} \left( \mathcal{M} \left( \mathscr{P}^{\Gamma}_{{\rm inv}} \left( \Omega_{{\rm DS}} \right), \sigma^{\Gamma} \right) \right)$, where $\mathcal{M} \left( \mathscr{P}^{\Gamma}_{{\rm inv}} \left( \Omega_{{\rm DS}} \right), \sigma^{\Gamma} \right)$ denotes the space of $\sigma^{\Gamma}$-invariant Borel probability measures on the orbit space $\mathscr{P}^{\Gamma}_{{\rm inv}} \left( \Omega_{{\rm DS}} \right)$. As $\left( \Pi_{0} \right)_{*}$ is continuous, the space $\mathscr{S}^{\Gamma}$ is compact. 

We now define the entropy of the holomorphic correspondence with respect to a measure $\nu \in \mathscr{S}^{\Gamma}$. First, recall from \cite{wal:1982} that the entropy of a finite measurable partition, say $\mathcal{P} = \left\{ P_{1}, P_{2}, \cdots, P_{s} \right\}$ of $\mathscr{P}^{\Gamma}_{{\rm inv}} \left( \Omega_{{\rm DS}} \right)$ with respect to a Borel probability measure $\mu$ supported on $\mathscr{P}^{\Gamma}_{{\rm inv}} \left( \Omega_{{\rm DS}} \right)$ is given by $H_{\mu} \left( \mathcal{P} \right) = - \sum\limits_{j\, =\, 1}^{s} \mu \left( P_{j} \right) \log \mu \left( P_{j} \right)$. Now, let $\nu \in \mathscr{S}^{\Gamma}$ and let $\mathcal{Q} = \left\{ Q_{1}, Q_{2}, \cdots, Q_{s} \right\}$ be a finite $\nu$-measurable partition of $\Omega_{{\rm DS}}$. Let $\mu \in \mathcal{M} \left( \mathscr{P}^{\Gamma}_{{\rm inv}} \left( \Omega_{{\rm DS}} \right), \sigma^{\Gamma} \right)$ be such that $\left( \Pi_{0} \right)_{*} \mu = \nu$. For $k \ge 0$ and $k \ge 1$, denote by $\mathcal{R}_{k}$ and $\mathcal{S}_{k}$ the finite $\mu$-measurable partitions of $\mathscr{P}^{\Gamma}_{{\rm inv}} \left( \Omega_{{\rm DS}} \right)$ given by $\left\{ \Pi_{k}^{-1} \left( Q_{1} \right), \Pi_{k}^{-1} \left( Q_{2} \right), \cdots, \Pi_{k}^{-1} \left( Q_{s} \right) \right\}$ and $\left\{ {\rm Proj}_{k}^{-1} \left( \{1\} \right), {\rm Proj}_{k}^{-1} \left( \{2\} \right), \cdots, {\rm Proj}_{k}^{-1} \left( \{N\} \right) \right\}$ respectively, where $N$ is the number of irreducible subvarieties in the representation of the correspondence $\Gamma$, as written in Definition \ref{defn:holocorr}. Define for $k \ge 0,\ \mathcal{Q}_{k}^{\Gamma} := \mathcal{R}_{k} \vee \mathcal{S}_{k + 1}$, the join of $\mathcal{R}_{k}$ and $\mathcal{S}_{k + 1}$, which is again a finite $\mu$-measurable partition of $\mathscr{P}^{\Gamma}_{{\rm inv}} \left( \Omega_{{\rm DS}} \right)$. 

We define the entropy of the correspondence $\Gamma$ with respect to the pair $\left( \nu, \mu \right)$ and the partition $\mathcal{Q}$ to be 
\[ H_{\left( \nu, \mu \right)} \left( \mathcal{Q}, \Gamma \right)\ \ :=\ \ \lim\limits_{n\, \to\, \infty} \frac{1}{n} H_{\mu} \left( \bigvee_{p\, =\, 0}^{n - 1} \mathcal{Q}_{p}^{\Gamma} \right). \]
The intermediate entropy of $\Gamma$ with respect to the pair $\left( \nu, \mu \right)$ is defined as 
\[ h_{\left( \nu, \mu \right)} (\Gamma)\ \ =\ \ \sup\limits_{\mathcal{Q}} H_{\left( \nu, \mu \right)} \left( \mathcal{Q}, \Gamma \right), \] 
where the supremum is considered over all finite $\nu$-measurable partitions of $\Omega_{{\rm DS}}$. 

\begin{definition}
\label{defn:entropycorr}
Let $\Gamma$ be a holomorphic correspondence on $\widehat{\mathbb{C}}$ with $\Omega_{{\rm DS}}$ being the support of the Dinh-Sibony measure of $\Gamma$. Let $\nu \in \mathscr{S}^{\Gamma}$. Then the measure theoretic entropy of $\Gamma$ with respect to $\nu$ is defined as 
\[ h_{\nu} (\Gamma)\ \ =\ \ \sup \left\{ h_{\left( \nu, \mu \right)} (\Gamma)\; :\; \mu \in  \mathcal{M} \left( \mathscr{P}^{\Gamma}_{{\rm inv}} \left( \Omega_{{\rm DS}} \right), \sigma^{\Gamma} \right)\ \text{such that}\ \nu = \left( \Pi_{0} \right)_{*} \mu \right\}. \] 
\end{definition}

We now state a result from \cite{ss:2026} that shall come in handy, in the sequel. Note that the authors work with the spaces $\widehat{\mathbb{C}}$ and $\mathscr{P}^{\Gamma} \left( \widehat{\mathbb{C}} \right)$ in that paper. We state those results with the correspondence being restricted to $\Omega_{{\rm DS}}$ and hence work with the orbit space $\mathscr{P}^{\Gamma}_{{\rm inv}} \left( \Omega_{{\rm DS}} \right)$. It must be noted that the proofs of the following holds {\it mutatis mutandis} the methods mentioned in the proofs of the corresponding results in \cite{ss:2026}. This idea also holds for other results from \cite{ss:2026} that we shall state in this paper.

\begin{theorem} \cite{ss:2026}
\label{thm:entropyconnect}
Let $\Gamma$ be a holomorphic correspondence on $\widehat{\mathbb{C}}$ with $\Omega_{{\rm DS}}$ being the support of the Dinh-Sibony measure of $\Gamma$. Consider the space $\mathscr{P}^{\Gamma}_{{\rm inv}} \left( \Omega_{{\rm DS}} \right)$ of infinitely long forward orbits of points in $\Omega_{{\rm DS}}$ which lie entirely in $\Omega_{{\rm DS}}$. Let $\sigma^{\Gamma}$ denote the left shift map on this space. Then for all pairs of measures $\left( \nu, \mu \right)$ with $\nu \in \mathscr{S}^{\Gamma}$ and $\mu \in \mathcal{M} \left( \mathscr{P}^{\Gamma}_{{\rm inv}} \left( \Omega_{{\rm DS}} \right), \sigma^{\Gamma} \right)$ satisfying $\left( \Pi_{0} \right)_{*} \mu = \nu$, we have $h_{\left( \nu, \mu \right)} (\Gamma) = h_{\mu} \left( \sigma^{\Gamma} \right)$, where $h_{\mu} \left( \sigma^{\Gamma} \right)$ denotes the measure theoretic entropy of the map $\sigma^{\Gamma}$ with respect to the measure $\mu$. 
\end{theorem}

Now, we proceed to define the notion of pressure associated to holomorphic correspondences. For the remainder of this section, we work with a holomorphic correspondence $\Gamma$ on $\widehat{\mathbb{C}}$ whose Dinh-Sibony measure has its support denoted by $\Omega_{{\rm DS}}$. We mainly work with the set $\left( \mathscr{P}_{n}^{\Gamma} \right)_{{\rm inv}} \left( \Omega_{{\rm DS}} \right)$ of forward paths of length $n$ under the iterates of $\Gamma$ that lie entirely inside $\Omega_{{\rm DS}}$. 

\begin{definition}
Given $\epsilon > 0$ and $n \in \mathbb{Z}_{+}$, a set $\mathcal{F} \subseteq \left( \mathscr{P}_{n}^{\Gamma} \right)_{{\rm inv}} \left( \Omega_{{\rm DS}} \right)$ is said to be $(n, \epsilon)$-separated if for any distinct points $\mathfrak{X}^{+}_{n} \left( z_{0}; \boldsymbol{\alpha} \right)_{\boldsymbol{j}}, \mathfrak{X}^{+}_{n} \left( w_{0}; \boldsymbol{\beta} \right)_{\boldsymbol{k}} \in \mathcal{F}$, we have  
\begin{eqnarray*}
\text{either}\ \ & \rho \left( \Pi_{(r, n)} \big( \mathfrak{X}^{+}_{n} \left( z_{0}; \boldsymbol{\alpha} \right)_{\boldsymbol{j}} \big), \Pi_{(r, n)} \big( \mathfrak{X}^{+}_{n} \left( w_{0}; \boldsymbol{\beta} \right)_{\boldsymbol{k}} \big) \right) > \epsilon, & \text{for some}\ 0 \le r \le n, \\ 
\text{or}\ \ & {\rm Proj}_{(r, n)} \big( \mathfrak{X}^{+}_{n} \left( z_{0}; \boldsymbol{\alpha} \right)_{\boldsymbol{j}} \big) \ne  {\rm Proj}_{(r, n)} \big( \mathfrak{X}^{+}_{n} \left( w_{0}; \boldsymbol{\beta} \right)_{\boldsymbol{k}} \big), & \text{for some}\ 1 \le r \le n. 
\end{eqnarray*}
\end{definition}

For a continuous function $f \in \mathcal{C} \left( \Omega_{{\rm DS}}, \mathbb{R} \right)$ and for $\mathfrak{X}^{+}_{n} \left( z_{0}; \boldsymbol{\alpha} \right)_{\boldsymbol{j}} \in \left( \mathscr{P}_{n}^{\Gamma} \right)_{{\rm inv}} \left( \Omega_{{\rm DS}} \right)$, we denote by $\mathcal{Z}_{n} (f) \left( \mathfrak{X}^{+}_{n} \left( z_{0}; \boldsymbol{\alpha} \right)_{\boldsymbol{j}} \right)$, the sum given by $\displaystyle{\sum\limits_{j\, =\, 0}^{n - 1} f \left( \Pi_{(j, n)} \left( \mathfrak{X}^{+}_{n} \left( z_{0}; \boldsymbol{\alpha} \right)_{\boldsymbol{j}} \right) \right)}$. Now, for any $\epsilon > 0$, let 
\begin{equation}
\mathscr{R}_{n}^{\Gamma} (f, \epsilon)\ \ =\ \ \sup\limits_{\mathcal{F}} \left\{ \sum\limits_{\mathfrak{X}^{+}_{n} \left( z_{0}; \boldsymbol{\alpha} \right)_{\boldsymbol{j}}\, \in\, \mathcal{F}} e^{\mathcal{Z}_{n} (f) \left( \mathfrak{X}^{+}_{n} \left( z_{0}; \boldsymbol{\alpha} \right)_{\boldsymbol{j}} \right)} \right\},
\end{equation}
where the supremum runs over all $(n, \epsilon)$-separated subsets of $\left( \mathscr{P}_{n}^{\Gamma} \right)_{{\rm inv}} \left( \Omega_{{\rm DS}} \right)$. 

\begin{definition}
The topological pressure of the function $f \in \mathcal{C} \left( \Omega_{{\rm DS}}, \mathbb{R} \right)$ with respect to the holomorphic correspondence $\Gamma$, denoted by ${\rm Pr} (\Gamma, f)$ is given by 
\[ {\rm Pr} (\Gamma, f)\ \ =\ \ \lim\limits_{\epsilon\, \to\, 0} \left( \limsup\limits_{n\, \to\, \infty} \frac{1}{n} \log \mathscr{R}_{n}^{\Gamma} (f, \epsilon) \right). \] 
We further define the topological entropy of the correspondence to be $h_{{\rm top}} (\Gamma) = {\rm Pr} (\Gamma, 0)$.
\end{definition}

We now state a result from \cite{ss:2025} that gives a connection between the notions of topological pressure of real valued continuous functions, in the case of the holomorphic correspondence $\Gamma$ and that of the map $\sigma^{\Gamma}$.

\begin{theorem} \cite{ss:2025}
\label{thm:eqpres}
Let $\Gamma$ be a holomorphic correspondence on $\widehat{\mathbb{C}}$ with $\Omega_{{\rm DS}}$ being the support of the Dinh-Sibony measure of $\Gamma$. Then, for any $f \in \mathcal{C} \left( \Omega_{{\rm DS}}, \mathbb{R} \right)$, we have ${\rm Pr} (\Gamma, f) = {\rm Pr} \left( \sigma^{\Gamma}, f \circ \Pi_{0} \right)$, where the notation ${\rm Pr} \left( \sigma^{\Gamma}, f \circ \Pi_{0} \right)$ denotes the pressure of the continuous function $f \circ \Pi_{0}$ with respect to the map $\sigma^{\Gamma}$ defined on $\mathscr{P}^{\Gamma}_{{\rm inv}} \left( \Omega_{{\rm DS}} \right)$. 
\end{theorem}

The notions of entropy and pressure are related by the variational principle, which is now stated below. 

\begin{theorem} \cite{ss:2026}
Let $\Gamma$ be a holomorphic correspondence on $\widehat{\mathbb{C}}$ with $\Omega_{{\rm DS}}$ being the support of the Dinh-Sibony measure of $\Gamma$. Let $f \in \mathcal{C} \left( \Omega_{{\rm DS}}, \mathbb{R} \right)$. Then we have,
\begin{equation}
\label{eqn:varprin}
{\rm Pr} (\Gamma, f)\ \ =\ \ \sup\limits_{\nu\, \in\, \mathscr{S}^{\Gamma}} \left\{ h_{\nu} (\Gamma) + \int f\, \mathrm{d}\nu \right\}.
\end{equation}
In particular, we have $h_{{\rm top}} (\Gamma) = \sup \left\{h_{\nu} (\Gamma) : \nu \in \mathscr{S}^{\Gamma} \right\}$.
\end{theorem}

\begin{definition} 
A measure $\nu \in \mathscr{S}^{\Gamma}$, where the supremum in Equation \eqref{eqn:varprin} is attained is said to be an \emph{equilibrium state} for the function $f$. 
\end{definition} 

\section{Some topological results}
\label{sec:four} 

In this section, we focus on a specific class of holomorphic correspondences that we are interested to work with, in the upcoming sections with regards to the uniqueness of an equilibrium state and also prove some interesting properties of the same. First, we introduce some basic tools that shall be useful in the proof of Proposition \ref{prop:distexp}. For a holomorphic correspondence $\digamma$ on $\widehat{\mathbb{C}}$, let 
\[ \mathcal{O} (\digamma)\ \ :=\ \ \left\{ \left( z_{0}, z_{1}, z_{2}, \cdots \right)\ :\ z_{0} \in \widehat{\mathbb{C}},\ \text{with}\ z_{r} \in \digamma(z_{r - 1})\ \text{for all}\ r \ge 0 \right\}. \]
We define a metric $d_{\mathcal{O} (\digamma)}$ on the set $\mathcal{O} (\digamma)$ given by,
\[ d_{\mathcal{O} (\digamma)} \left( (z_{0}, z_{1}, z_{2}, \cdots), (w_{0}, w_{1}, w_{2}, \cdots) \right)\ \ =\ \ \sum\limits_{n\, \ge\, 0} \frac{1}{2^{n}} \frac{\rho \left( z_{n},\, w_{n} \right)}{\left( 1 + \rho \left( z_{n},\, w_{n} \right) \right)}. \] 

Define the left shift map $\sigma_{1} : \mathcal{O} (\digamma) \longrightarrow \mathcal{O} (\digamma)$ by $\sigma_{1} \left( ( z_{0}, z_{1}, z_{2}, z_{3}, \cdots ) \right) = ( z_{1}, z_{2}, z_{3}, \cdots )$. Now consider the one sided full shift space $\mathcal{S}^{\mathbb{Z}_{+}}$ on the symbol set $\mathcal{S} = \{ 1, 2, \cdots, N \}$, where $N$ is the number of varieties in the definition of the holomorphic correspondence $\Gamma$. We denote by $\sigma_{2}$ the shift map on the space $\mathcal{S}^{\mathbb{Z}_{+}}$. Let $\sigma := \sigma_{1} \times \sigma_{2}$ denote the map $\sigma : \mathcal{O} (\digamma) \times \mathcal{S}^{\mathbb{Z}_{+}} \longrightarrow \mathcal{O} (\digamma) \times \mathcal{S}^{\mathbb{Z}_{+}}$ given by $\sigma \left( \boldsymbol{z}, \boldsymbol{\alpha} \right) = \left( \sigma_{1} \left( \boldsymbol{z} \right), \sigma_{2} \left( \boldsymbol{\alpha} \right) \right)$. Under the usual product topology given to $\mathcal{O} (\digamma) \times \mathcal{S}^{\mathbb{Z}_{+}}$, the map $\sigma$ is continuous. Note that we have $\mathscr{P}^{\Gamma}_{{\rm inv}} \left( \Omega_{{\rm DS}} \right) \subseteq \mathscr{P}^{\Gamma} \left( \widehat{\mathbb{C}} \right) \subseteq \mathcal{O} (\digamma) \times \mathcal{S}^{\mathbb{Z}_{+}}$. Further, the metric that we have already defined on $\mathscr{P}^{\Gamma} \left( \widehat{\mathbb{C}} \right)$ generates the same topology as the one it obtains as a subspace of $\mathcal{O} (\digamma) \times \mathcal{S}^{\mathbb{Z}_{+}}$. Furthermore, the shift map associated to $\Gamma$, namely $\sigma^{\Gamma}$ on the space $\mathscr{P}^{\Gamma} \left( \widehat{\mathbb{C}} \right)$ ({\it resp.} $\mathscr{P}^{\Gamma}_{{\rm inv}} \left( \Omega_{{\rm DS}} \right)$) can be viewed as the restriction of the above map $\sigma$ to the space $\mathscr{P}^{\Gamma} \left( \widehat{\mathbb{C}} \right)$ ({\it resp.} $\mathscr{P}^{\Gamma}_{{\rm inv}} \left( \Omega_{{\rm DS}} \right)$). 

\subsection{Openness}

We now take a brief detour to mention a few ideas from complex geometry that we shall use in the proof of one of the main results in this section, namely, Theorem \ref{thm:open}. For Theorem \ref{thm:rem}, we use the presentation followed in \cite{lon:2022}, while the original result is actually from \cite{rr:1957}. We use the following notation in the statement of Theorem \ref{thm:rem}. 

Consider a holomorphic map $T : X_{1} \longrightarrow X_{2}$ where $X_{1}$ and $X_{2}$ are complex manifolds of dimension $k_{1}$ and $k_{2}$ respectively. Let $\mathcal{A} \subseteq X_{1}$ be a complex subvariety of pure dimension. Consider the restriction $T\vert_{\mathcal{A}}$ of the function $T$ to the subvariety $\mathcal{A}$. It is known that for any $z \in \mathcal{A}$, we have $\left( T\vert_{\mathcal{A}} \right)^{-1} T\vert_{\mathcal{A}} (z) := T^{-1} (T(z)) \cap \mathcal{A}$ is a complex subvariety of $X_{1}$ and hence, we can speak of its dimension. We now define the rank of $T\vert_{\mathcal{A}}$ at some $z \in \mathcal{A}$ to be ${\rm rank}_{z} T\vert_{\mathcal{A}} := \dim_{z} \mathcal{A} - \dim_{z} \left( T\vert_{\mathcal{A}} \right)^{-1} T\vert_{\mathcal{A}} (z)$.

\begin{theorem} \cite{rr:1957}
\label{thm:rem}
Let $T : X_{1} \longrightarrow X_{2}$ be a holomorphic map and $\mathcal{A}$ be a complex subvariety of $X_{1}$ of pure dimension. Then $T\vert_{\mathcal{A}}$ is an open map if and only if ${\rm rank}_{z} T\vert_{\mathcal{A}} = \dim (X_{2})$ for all $z \in \mathcal{A}$. 
\end{theorem}

\begin{theorem}
\label{thm:open}
Let $\digamma$ be a holomorphic correspondence on $\widehat{\mathbb{C}}$ and let $U$ be an open subset in $\widehat{\mathbb{C}}$. Then, the image set $\digamma(U)$ is also open in $\widehat{\mathbb{C}}$. That is, every holomorphic correspondence on $\widehat{\mathbb{C}}$ maps open sets to open sets.  
\end{theorem}

\begin{proof}
Recall that $\digamma(U) = \pi_{2} \left( \pi_{1}^{-1} (U) \cap |\Gamma| \right)$ where $|\Gamma| = \bigcup\limits_{j\, =\, 1}^{N} \Gamma_{j}$. Here, we consider the variety representation of the correspondence $\digamma$ given by $\Gamma =  \sideset{}{'} \sum\limits_{j\, =\, 1}^{N} \Gamma_{j}$. Since $\pi_{1}$ is continuous, we have $\pi_{1}^{-1} (U) \cap |\Gamma|$ as an open set in $|\Gamma|$. 
    
Now, we are in a position to apply Theorem \ref{thm:rem} and hence choose $X_{1} = \widehat{\mathbb{C}} \times \widehat{\mathbb{C}},\ X_{2} = \widehat{\mathbb{C}},\ \mathcal{A} = |\Gamma|$ and $T = \pi_{2}$, in accordance with the notations of Theorem \ref{thm:rem}. As $\digamma$ is a holomorphic correspondence, we have the set $\left\{ w \in \widehat{\mathbb{C}} : \dim \left( \pi_{2}^{-1} \left( \{w\} \right) \cap |\Gamma| \right) > 0 \right\} = \emptyset$. Thus, we have $\dim_{(z, w)} \left( \pi_{2}\vert_{|\Gamma|} \right)^{-1} \pi_{2}\vert_{|\Gamma|} (z, w) = 0$ for all $(z, w) \in |\Gamma|$. Further, since we have $\dim_{(z, w)} (|\Gamma|) = 1$, we get ${\rm rank}_{(z, w)} \pi_{2}\vert_{\mathcal{A}} = 1 = \dim \widehat{\mathbb{C}}$ for all $(z, w) \in |\Gamma|$. 
    
An application of Theorem \ref{thm:rem} now gives $\digamma(U) = \pi_{2} \left( \pi_{1}^{-1} (U) \cap |\Gamma| \right)$ is open in $\widehat{\mathbb{C}}$. 
\end{proof}

It can be easily observed by using Theorem \ref{thm:rem} that the above theorem holds for holomorphic correspondences defined on arbitrary compact complex manifolds as well. Before proceeding to our next result, we fix a notation. Let $W \subseteq \mathcal{S}$ and $X, Y, Z \subseteq \widehat{\mathbb{C}}$. We define,
\[ I_{(X, Y, Z)} (W) := \bigg\{ \beta \in \mathcal{S} : \exists \alpha \in W, z_{0} \in X, z_{1} \in Y, z_{2} \in Z\ \text{such that}\ (z_{0}, z_{1}) \in \Gamma_{\alpha}, (z_{1}, z_{2}) \in \Gamma_{\beta} \bigg\}. \]

\begin{proposition}
\label{prop:open}
Let $\digamma$ be a holomorphic correspondence on $\widehat{\mathbb{C}}$. Then the shift map $\sigma^{\Gamma}: \mathscr{P}^{\Gamma}_{{\rm inv}} \left( \Omega_{{\rm DS}} \right) \longrightarrow \mathscr{P}^{\Gamma}_{{\rm inv}} \left( \Omega_{{\rm DS}} \right)$ is an open map.
\end{proposition} 

\begin{proof}
Let $\mathcal{U}$ be an open set in $\mathscr{P}^{\Gamma}_{{\rm inv}} \left( \Omega_{{\rm DS}} \right)$ and let $\mathfrak{X}^{+} \left( z_{1}; \boldsymbol{\alpha} \right)_{\boldsymbol{j}} \in \sigma^{\Gamma} (\mathcal{U})$. We show that there exists a basic open set of $\mathscr{P}^{\Gamma}_{{\rm inv}} \left( \Omega_{{\rm DS}} \right)$ that contains $\mathfrak{X}^{+} \left( z_{1}; \boldsymbol{\alpha} \right)_{\boldsymbol{j}}$ and is contained in $\sigma^{\Gamma} (\mathcal{U})$. Let $\mathfrak{X}^{+} \left( z_{0}; \alpha_{0}\boldsymbol{\alpha} \right)_{j_{0}\boldsymbol{j}} \in \mathcal{U}$ such that $\sigma^{\Gamma} (\mathfrak{X}^{+} \left( z_{0}; \alpha_{0}\boldsymbol{\alpha} \right)_{j_{0}\boldsymbol{j}}) = \mathfrak{X}^{+} \left( z_{1}; \boldsymbol{\alpha} \right)_{\boldsymbol{j}}$. As $\mathcal{U}$ is open in $\mathscr{P}^{\Gamma}_{{\rm inv}} \left( \Omega_{{\rm DS}} \right)$, there exists $r, s \in \mathbb{Z}_{+}$ and open sets $V_{m} \subset \widehat{\mathbb{C}}$ and $W_{n} \subset \mathcal{S}$ for $1 \le m \le r$ and $1 \le n \le s$ respectively such that,
\[ \mathfrak{X}^{+} \left( z_{0}; \alpha_{0}\boldsymbol{\alpha} \right)_{j_{0}\boldsymbol{j}} \in \left( \left( V_{1} \times \cdots \times V_{r} \times \widehat{\mathbb{C}}^{\mathbb{Z}_{+}} \right) \times \left( W_{1} \times \cdots \times W_{s} \times \mathcal{S}^{\mathbb{Z}_{+}} \right) \right) \cap \mathscr{P}^{\Gamma}_{{\rm inv}} \left( \Omega_{{\rm DS}} \right) \subseteq \mathcal{U}. \] 
This results in 
\[ \mathfrak{X}^{+} \left( z_{1}; \boldsymbol{\alpha} \right)_{\boldsymbol{j}} \in \sigma^{\Gamma} \left( \left( V_{1} \times \cdots \times V_{r} \times \widehat{\mathbb{C}}^{\mathbb{Z}_{+}} \right) \times \left( W_{1} \times \cdots \times W_{s} \times \mathcal{S}^{\mathbb{Z}_{+}} \right) \right) \cap \mathscr{P}^{\Gamma}_{{\rm inv}} \left( \Omega_{{\rm DS}} \right) \subseteq \sigma^{\Gamma} (\mathcal{U}) \] 
and hence we have the point $\mathfrak{X}^{+} \left( z_{1}; \boldsymbol{\alpha} \right)_{\boldsymbol{j}}$ to be a member in the intersection of $\mathscr{P}^{\Gamma}_{{\rm inv}} \left( \Omega_{{\rm DS}} \right)$ with 
\[ \left( \left( \left( \digamma(V_{1}) \cap V_{2} \right) \times V_{3} \times \cdots \times V_{r} \times \widehat{\mathbb{C}}^{\mathbb{Z}_{+}} \right) \times \left( \left( I_{\left( V_{1}, V_{2}, V_{3} \right)} (W_{1}) \cap W_{2} \right) \times W_{3} \times \cdots \times W_{s} \times \mathcal{S}^{\mathbb{Z}_{+}} \right) \right). \] 

By Theorem \ref{thm:open}, we have $\digamma(V_{1})$ is open and due to discrete topology on $\mathcal{S}$, the set $I_{\left( V_{1}, V_{2}, V_{3} \right)}(W_{1})$ is also open. This is the open neighbourhood of $\mathfrak{X}^{+} \left( z_{1}; \boldsymbol{\alpha} \right)_{\boldsymbol{j}}$ contained in $\sigma^{\Gamma} (\mathcal{U})$ that we were looking for. Hence the result follows.
\end{proof}

\subsection{Topological transitivity and distance expanding conditions}

We begin this subsection with a result that emphasizes the importance of working with the domain $\Omega_{{\rm DS}}$. This will turn out to be one of the crucial ingredients of our main result concerning the uniqueness of equilibrium measure. Had we worked with the domain of the correspondence being $\widehat{\mathbb{C}}$ in our main result, we will have to impose an extra hypothesis related to the concept of topological transitivity. But restricting the domain to $\Omega_{{\rm DS}}$ gives this for free. 

\begin{lemma}
\label{lem:toptrans}
Let $\Gamma$ be a holomorphic correspondence on $\widehat{\mathbb{C}}$ with $\Omega_{{\rm DS}}$ being the support of the Dinh-Sibony measure of $\Gamma$. Then the map $\sigma^{\Gamma}$ is topologically transitive on $\mathscr{P}^{\Gamma}_{{\rm inv}} \left( \Omega_{{\rm DS}} \right)$.
\end{lemma}

\begin{proof}
We prove that given non-empty open sets $U_{1}$ and $U_{2}$ of $\mathscr{P}^{\Gamma}_{{\rm inv}} \left( \Omega_{{\rm DS}} \right)$, there exists $M \in \mathbb{Z}_{+}$ such that $\left( \sigma^{\Gamma} \right)^{M} (U_{1}) \cap U_{2} \ne \emptyset$. 

For $\mathfrak{X}^{+} \left( z_{0}; \boldsymbol{\alpha} \right)_{\boldsymbol{j}} \in \mathscr{P}^{\Gamma}_{{\rm inv}} \left( \Omega_{{\rm DS}} \right)$, we know that $\mathcal{O}^{-} \left( \mathfrak{X}^{+} \left( z_{0}; \boldsymbol{\alpha} \right)_{\boldsymbol{j}} \right) := \bigcup\limits_{n\, \ge\, 1} \left( \sigma^{\Gamma} \right)^{-n} \left( \mathfrak{X}^{+} \left( z_{0}; \boldsymbol{\alpha} \right)_{\boldsymbol{j}} \right)$ is dense in $\mathscr{P}^{\Gamma}_{{\rm inv}} \left( \Omega_{{\rm DS}} \right)$. Interested readers may refer to Lemma 10.1 in \cite{ss:2025} for a proof of the same. 

Now let $\mathfrak{X}^{+} \left( z_{0}; \boldsymbol{\alpha} \right)_{\boldsymbol{j}} \in U_{2}$. As $\mathcal{O}^{-} \left( \mathfrak{X}^{+} \left( z_{0}; \boldsymbol{\alpha} \right)_{\boldsymbol{j}} \right)$ is dense in $\mathscr{P}^{\Gamma}_{{\rm inv}} \left( \Omega_{{\rm DS}} \right)$, there exists $M \in \mathbb{Z}_{+}$ such that $\mathfrak{X}^{+} \left( z_{j_{-M}}^{(M)}; \boldsymbol{\alpha' \alpha} \right)_{\boldsymbol{j' j}} \in \mathcal{O}^{-} \left( \mathfrak{X}^{+} \left( z_{0}; \boldsymbol{\alpha} \right)_{\boldsymbol{j}} \right) \cap U_{1}$. Here, the symbol $\mathfrak{X}^{+} \left( z_{j_{-M}}^{(-M)}; \boldsymbol{\alpha' \alpha} \right)_{\boldsymbol{j' j}}$ must be interpreted similarly to the explanation given for the notation $\mathfrak{X}^{+} \left( z_{0}; \boldsymbol{\alpha} \right)_{\boldsymbol{j}}$, in Section \ref{sec:two}. Furthermore, in this case, $\boldsymbol{\alpha'}$ is a suitable $M$-tuple with symbols from $\mathcal{S}$ and $\boldsymbol{j'}$ is a suitable $M$-tuple with symbols from the set $\{1, 2, \cdots, d_{{\rm top}}\}$, where $d_{{\rm top}}$ denotes the topological degree of $\digamma$. As we have $\left( \sigma^{\Gamma} \right)^{M} \left( \mathfrak{X}^{+} \left( z_{j_{-M}}^{(M)}; \boldsymbol{\alpha' \alpha} \right)_{\boldsymbol{j' j}} \right) = \mathfrak{X}^{+} \left( z_{0}; \boldsymbol{\alpha} \right)_{\boldsymbol{j}}$, the proof of the lemma is now complete.    
\end{proof}

The remainder of the section deals with distance expanding correspondences. We first give the definition of this notion and then prove a result that shall be useful in the next section. 

\begin{definition}
Let $\digamma$ be a holomorphic correspondence on $\widehat{\mathbb{C}}$. The correspondence $\digamma$ is said to be \emph{distance expanding} on $\widehat{\mathbb{C}}$ if there exists $\lambda > 1,\ n \in \mathbb{Z}_{+}$ and $\eta > 0$ such that for all $z, w \in \widehat{\mathbb{C}}$ with $\rho (z, w) \le \eta$, we have $\displaystyle{\inf \big\{ \rho (z_{1}, w_{1}) : z_{1} \in \digamma(z_{0}), w_{1} \in \digamma(w_{0}) \big\} \ge \lambda\, \rho(z_{0}, w_{0})}$. 
\end{definition}

\begin{proposition}
\label{prop:distexp} 
Let $\digamma$ be a distance expanding holomorphic correspondence on $\widehat{\mathbb{C}}$. Then the corresponding shift map $\sigma^{\Gamma}: \mathscr{P}^{\Gamma}_{{\rm inv}} \left( \Omega_{{\rm DS}} \right) \longrightarrow \mathscr{P}^{\Gamma}_{{\rm inv}} \left( \Omega_{{\rm DS}} \right)$ is also distance expanding.
\end{proposition} 

\begin{proof}
We initially observe that in order to prove this proposition, it is sufficient to prove that the map $\sigma = \sigma_{1} \times \sigma_{2}$ as defined at the beginning of this current section is distance expanding. This is because the restriction of a distance expanding map to an invariant subspace is again distance expanding and hence $\sigma^{\Gamma}$ would then turn out to be distance expanding on $\mathscr{P}^{\Gamma}_{{\rm inv}} \left( \Omega_{{\rm DS}} \right)$. Also, to prove that $\sigma$ is distance expanding, it is sufficient to prove that $\sigma_{1}$ and $\sigma_{2}$ are distance expanding on their appropriate domains. As $\sigma_{2}$ is the shift map on the space $\mathcal{S}^{\mathbb{Z}_{+}}$, it is an easy exercise to prove that it is distance expanding. Now all that remains to be proved is that $\sigma_{1}$ is distance expanding on $\mathcal{O} (\digamma)$. This follows \emph{mutatis mutandis} the proof of Proposition 7.8 in \cite{llz:2023} and using the hypothesis that $\digamma$ is distance expanding on $\widehat{\mathbb{C}}$. 
\end{proof}

\section{Proofs of the main theorems} 
\label{sec:five}

In this section, we prove our main theorem on the existence of a unique equilibrium state corresponding to any real valued H\"{o}lder continuous function defined on $\Omega_{{\rm DS}}$. Before doing so, we state a theorem that can be found in \cite{pu:2010} that shall be useful for us in proving our main theorem. See Theorem 4.6.2 in the indicated reference for a proof of the same. 

\begin{theorem}
\label{thm:mapunique}
Let $T : X \longrightarrow X$ be an open, distance expanding, topologically transitive and continuous map on a compact metric space $X$. Then there exists a unique equilibrium state for every H\"{o}lder continuous potential $\phi \in \mathcal{C} (X, \mathbb{R})$. 
\end{theorem} 

\subsection{Proof of Theorem \ref{thm:uniqueness}}

We are now ready to prove the existence of a unique equilibrium state associated to a H\"{o}lder continuous potential. 

\begin{proof}[of Theorem \ref{thm:uniqueness}] 

We first make an observation on the entropy $h_{\nu} (\Gamma)$, for $\nu \in \mathscr{S}^{\Gamma}$. Using the hypothesis in the statement of the theorem and Proposition \ref{prop:distexp}, we have $\sigma^{\Gamma}$ is distance expanding on $\mathscr{P}^{\Gamma}_{{\rm inv}} \left( \Omega_{{\rm DS}} \right)$. Thus by Theorem 3.1.1 in \cite{pu:2010}, $\sigma^{\Gamma}$ is a forward expansive map. By Theorem 2.5.6 in \cite{pu:2010}, this results in the map $\mu \mapsto h_{\mu} (\sigma^{\Gamma})$ being upper semi-continuous on $\mathcal{M} \left( \mathscr{P}^{\Gamma}_{{\rm inv}} \left( \Omega_{{\rm DS}} \right), \sigma^{\Gamma} \right)$. 

For $\nu \in \mathscr{S}^{\Gamma}$, we have by Definition \ref{defn:entropycorr} and Theorem \ref{thm:entropyconnect},
\begin{eqnarray}
\label{eqn:entropyuniqmeasure}
h_{\nu} (\Gamma) & = & \sup \left\{ h_{(\nu, \mu)} (\Gamma)\ :\ \mu \in \mathcal{M} \left( \mathscr{P}^{\Gamma}_{{\rm inv}} \left( \Omega_{{\rm DS}} \right), \sigma^{\Gamma} \right) \cap \left( \left( \Pi_{0} \right)_{*} \right)^{-1} (\nu) \right\} \nonumber \\ 
& = & \sup \left\{ h_{\mu} (\sigma^{\Gamma}): \mu \in \mathcal{M} \left( \mathscr{P}^{\Gamma}_{{\rm inv}} \left( \Omega_{{\rm DS}} \right), \sigma^{\Gamma} \right) \cap \left( \left( \Pi_{0} \right)_{*} \right)^{-1} (\nu) \right\}. 
\end{eqnarray}

Since $\mathcal{M} \left( \mathscr{P}^{\Gamma}_{{\rm inv}} \left( \Omega_{{\rm DS}} \right), \sigma^{\Gamma} \right) \cap \left( \left( \Pi_{0} \right)_{*} \right)^{-1} (\nu)$ is a compact set and the entropy map is upper semi-continuous, the supremum in Equation \eqref{eqn:entropyuniqmeasure} is attained. 

We now proceed to the next step of the proof. Let $f \in \mathcal{C} \left( \Omega_{{\rm DS}}, \mathbb{R} \right)$ be some H\"{o}lder continuous function. Consider the H\"{o}lder continuous function $F = f \circ \Pi_{0}$. By using Proposition \ref{prop:open}, Proposition \ref{prop:distexp}, Lemma \ref{lem:toptrans} and Theorem \ref{thm:mapunique}, we infer that there exists a unique measure $\mu_{F} \in \mathcal{M} \left( \mathscr{P}^{\Gamma}_{{\rm inv}} \left( \Omega_{{\rm DS}} \right), \sigma^{\Gamma} \right)$ such that
\begin{equation}
\label{eqn:varprinmap}
{\rm Pr} \left( \sigma^{\Gamma}, F \right)\ \ =\ \ h_{\mu_{F}} \left( \sigma^{\Gamma} \right) + \int\limits_{\mathscr{P}^{\Gamma}_{{\rm inv}} \left( \Omega_{{\rm DS}} \right)} F \mathrm{d}\mu_{F}. 
\end{equation} 

Set $\nu_{f} = \left( \Pi_{0} \right)_{*} \mu_{F}$. Now let $\mu' \in \mathcal{M} \left( \mathscr{P}^{\Gamma}_{{\rm inv}} \left( \Omega_{{\rm DS}} \right), \sigma^{\Gamma} \right) \cap \left( \left( \Pi_{0} \right)_{*} \right)^{-1} \left( \nu_{f} \right)$. Clearly we have $\displaystyle{\int \left( f \circ \Pi_{0} \right) \mathrm{d}\mu' = \int \left( f \circ \Pi_{0} \right) \mathrm{d}\mu_{F}}$. This, when combined with the expression for ${\rm Pr} \left( \sigma^{\Gamma}, F \right)$ obtained using the variational principle under the dynamics of $\sigma^{\Gamma}$ and Equation \eqref{eqn:varprinmap} gives 
\begin{equation}
\label{eqn:entropycorrmap}
h_{\nu_{f}} \left( \Gamma \right)\ \ =\ \ h_{\mu_{F}} \left( \sigma^{\Gamma} \right). 
\end{equation}

We now prove that $\nu_{f}$ is the required unique equilibrium state for $f$. Observe that we have 
\begin{equation}
\label{eqn:intgeq}
\int\limits_{\mathscr{P}^{\Gamma}_{{\rm inv}} \left( \Omega_{{\rm DS}} \right)} F \mathrm{d}\mu_{F}\ \ =\ \ \int\limits_{\mathscr{P}^{\Gamma}_{{\rm inv}} \left( \Omega_{{\rm DS}} \right)} \left( f \circ \Pi_{0} \right) \mathrm{d}\mu_{F}\ \ =\ \ \int\limits_{\Omega_{{\rm DS}}} f \mathrm{d} \left( \left( \Pi_{0} \right)_{*} \mu_{F} \right)\ \ =\ \ \int\limits_{\Omega_{{\rm DS}}} f \mathrm{d} \nu_{f}.
\end{equation}

By using Theorem \ref{thm:eqpres} and Equations \eqref{eqn:entropycorrmap} and \eqref{eqn:intgeq} in Equation \eqref{eqn:varprinmap}, we obtain 
\[ {\rm Pr} (\Gamma, f)\ \ =\ \ h_{\nu_{f}} (\Gamma) + \int\limits_{\Omega_{{\rm DS}}} f \mathrm{d}\nu_{f}, \]
which implies that $\nu_{f}$ is an equilibrium state for $f$. 

To prove that it is the unique equilibrium state, let us assume $\nu \in \mathscr{S}^{\Gamma}$ to be another measure (possibly different from $\nu_{f}$) such that $\displaystyle{{\rm Pr} (\Gamma, f) = h_{\nu} (\Gamma) + \int f \mathrm{d}\nu}$. Let $\mu_{1} \in \mathcal{M} \left( \mathscr{P}^{\Gamma}_{{\rm inv}} \left( \Omega_{{\rm DS}} \right), \sigma^{\Gamma} \right)$ be such that $\left( \Pi_{0} \right)_{*} \mu_{1} = \nu$ and $h_{\nu} (\Gamma) = h_{\mu_{1}} \left( \sigma^{\Gamma} \right)$. Suppose $\mu_{1} = \mu_{F}$, then we have $\left( \Pi_{0} \right)_{*} \mu_{1} = \left( \Pi_{0} \right)_{*} \mu_{F}$ and thus, we are done. If not, then we have 
\[ {\rm Pr} (\Gamma, f)\ \ =\ \ h_{\nu} (\Gamma) + \int f \mathrm{d}\nu\ \ =\ \ h_{\nu_{f}} (\Gamma) + \int f \mathrm{d}\nu_{f}. \] 
This gives 
\[ h_{\mu_{1}} \left( \sigma^{\Gamma} \right) + \int \left( f \circ \Pi_{0} \right) \mathrm{d}\mu_{1}\ \ =\ \ h_{\mu_{F}} \left( \sigma^{\Gamma} \right) + \int \left( f \circ \Pi_{0} \right) \mathrm{d}\mu_{F}\ \ =\ \ {\rm Pr} \left( \sigma^{\Gamma}, f \circ \Pi_{0} \right). \] 

Now, using the uniqueness of the equilibrium state for the H\"{o}lder continuous function $f \circ \Pi_{0}$ under the dynamics of $\sigma^{\Gamma}$, we have that $\mu_{1} = \mu_{F}$, which implies $\nu_{f} = \nu$. Thus, $f$ has a unique equilibrium state.
\end{proof}

\subsection{Some results related to the Ruelle operator}
\label{subsec} 

We invesigate and learn a few properties of the Ruelle operator associated with holomorphic correspondences, which was studied in \cite{ss:2025}. Our aim here is to connect the Ruelle operator and the unique equilibrium state obtained in Theorem \ref{thm:uniqueness} whenever possible. We first begin with some basic definitions. 

Let $f \in \mathcal{C} \left( \Omega_{{\rm DS}}, \mathbb{R} \right)$. Define an operator $\mathcal{L}_{f} : \mathcal{C} \left( \Omega_{{\rm DS}}, \mathbb{R} \right) \longrightarrow \mathcal{C} \left( \Omega_{{\rm DS}}, \mathbb{R} \right)$ whose action on points in $\Omega_{{\rm DS}}$ is given by 
\[ \left( \mathcal{L}_{f} (g) \right) (z)\ \ =\ \ \sum\limits_{w\, \in\, \digamma^{\dagger} (z)} e^{f(w)} g(w). \] 
$\mathcal{L}_{f}$ is called the \emph{Ruelle operator associated to $f$} and its dual linear map defined on the dual space $\mathcal{C} \left( \Omega_{{\rm DS}}, \mathbb{R} \right)^{*}$ is given by 
\[ \left( \mathcal{L}_{f}^{*} (m) \right) (g)\ \ =\ \ \int\limits_{\Omega_{{\rm DS}}} \mathcal{L}_{f}(g) \mathrm{d}m\ \ \ \text{where}\ \ \ g \in \mathcal{C} \left( \Omega_{{\rm DS}}, \mathbb{R} \right). \] 

For a continuous map $T$ defined on a compact metric space $X$ and a function $\phi \in \mathcal{C} \left( X, \mathbb{R} \right)$, we denote the Ruelle operator associated to $\phi$ by $L_{\phi}$ defined on $\mathcal{C} \left( X, \mathbb{R} \right)$, with its action given by $\displaystyle{L_{\phi} (\psi) (x) = \sum\limits_{y\, \in\, T^{-1} (x)} e^{\phi(y)} \psi(y)}$ where $x \in X$. The dual of $L_{\phi}$ is denoted by $L_{\phi}^{*}$. We now state the following theorem regarding these operators, which is a combination of Theorem 13.6.2 and Proposition 13.6.16 in \cite{mru:2022}.

\begin{theorem}
\label{thm:measureexistsruelle}
Let $T : X \longrightarrow X$ be an open, distance expanding, topologically transitive, continuous surjection on a compact metric space $X$. Suppose $\phi \in \mathcal{C} \left( X, \mathbb{R} \right)$ is H\"{o}lder continuous, then there exists $m_{\phi} \in \mathcal{M} (X)$ such that $L_{\phi}^{*} m_{\phi} = e^{{\rm Pr} (T, \phi)} m_{\phi}$. 
\end{theorem}

Let $T : X \longrightarrow X$ and $\phi \in \mathcal{C} \left( X, \mathbb{R} \right)$ satisfy the hypothesis in Theorem \ref{thm:measureexistsruelle}, pertaining to which $L_{\phi}$ is a Ruelle operator. Taking cue from Haydn (refer \cite{hay:1999}), one may study the normalised Ruelle operator, denoted by $\widehat{L_{\phi}}$, also defined on $\mathcal{C} \left( X, \mathbb{R} \right)$ and given by $\widehat{L_{\phi}} : = e^{- {\rm Pr} (T, \phi)} L_{\phi}$. In fact, from the results in \cite{mru:2022}, it turns out that $\widehat{L_{\phi}}$ can be extended to the space of $m_{\phi}$-integrable functions, $L^{1} (m_{\phi})$. With these notations in the bag, we have the following result which is a combination of Theorem 13.7.7 and Proposition 13.7.12 from \cite{mru:2022}. 

\begin{theorem}
\label{thm:mapseigenfn}
Let $T : X \longrightarrow X$ be an open, distance expanding, topologically transitive and continuous map on a compact metric space $X$. Suppose $\phi \in \mathcal{C} \left( X, \mathbb{R} \right)$ is a H\"{o}lder continuous function and $h_{\phi}$ denotes a fixed point of $\widehat{L_{\phi}}$ with $h_{\phi} \ge 0$ and $\displaystyle{\int h_{\phi} \mathrm{d}m_{\phi} = 1}$. Then, the measure $\mu_{\phi} = h_{\phi} m_{\phi}$ is the unique $T$-invariant equilibrium state for $\phi$.
\end{theorem} 

A further normalisation of $L_{\phi}$ is also possible by adding a coboundary, as follows: For $T : X \longrightarrow X$ and $\phi \in \mathcal{C} \left( X, \mathbb{R} \right)$ satisfying the hypothesis in Theorem \ref{thm:mapseigenfn}, define 
\begin{equation}
\label{eqn:normalmaps}
\hat{\phi}\ \ :=\ \ \phi + \log h_{\phi} - \log \left( h_{\phi} \circ T \right). 
\end{equation}

If $\hat{\phi}$ is H\"{o}lder continuous, then the measure $m_{\hat{\phi}}$ satisfying $L_{\hat{\phi}}^{*} m_{\hat{\phi}} = e^{{\rm Pr} (T, \hat{\phi})} m_{\hat{\phi}}$ is the unique equilibrium state for $\hat{\phi}$. Interested readers can refer \cite{mru:2022} for more details. 

Motivated by this theory for maps, we study the connection between the unique equilibrium state of a H\"{o}lder continuous potential defined on $\Omega_{{\rm DS}}$ and the eigenmeasure associated to the dual of its Ruelle operator, in the setting of a holomorphic correspondence $\Gamma$ on $\widehat{\mathbb{C}}$. We start towards that with a small observation. 

As before, let $\Gamma$ be a holomorphic correspondence on $\widehat{\mathbb{C}}$ and let $\Omega_{{\rm DS}}$ be the support of the Dinh-Sibony measure of $\Gamma$. Consider $f \in \mathcal{C} \left( \Omega_{{\rm DS}}, \mathbb{R} \right)$. Set $F = f \circ \Pi_{0}$. Suppose there exists a measure $\mu_{F}$ such that $L_{F}^{*} \mu_{F} = \Lambda_{F} \mu_{F}$ for some $\mu_{F} \in \mathcal{M} \left( \mathscr{P}^{\Gamma}_{{\rm inv}} \left( \Omega_{{\rm DS}} \right), \sigma^{\Gamma} \right)$ and $\Lambda_{F} \in \mathbb{R}$. Denote by $\nu_{f}$ the measure $\left( \Pi_{0} \right)_{*} \mu_{F}$. Then we have $\mathcal{L}_{f}^{*} \nu_{f} = \Lambda_{F} \nu_{f}$. This can be easily seen, as we have for any $g \in \mathcal{C} \left( \Omega_{{\rm DS}}, \mathbb{R} \right)$, 
\begin{eqnarray*}
\left( \mathcal{L}_{f}^{*} \nu_{f} \right) (g) & = & \hspace{+0.4cm} \int\limits_{\Omega_{{\rm DS}}} \mathcal{L}_{f} (g) \mathrm{d}\nu_{f}\ \ \ \ \hspace{+1.6cm} =\ \ \int\limits_{\mathscr{P}^{\Gamma}_{{\rm inv}} \left( \Omega_{{\rm DS}} \right)} \left( \mathcal{L}_{f} (g) \right) \circ \Pi_{0} \mathrm{d}\mu_{F} \\ 
& = & \int\limits_{\mathscr{P}^{\Gamma}_{{\rm inv}} \left( \Omega_{{\rm DS}} \right)} L_{F} \left( g \circ \Pi_{0} \right) \mathrm{d}\mu_{F}\ \ \ \ =\ \ \int\limits_{\mathscr{P}^{\Gamma}_{{\rm inv}} \left( \Omega_{{\rm DS}} \right)} \Lambda_{F} \left( g \circ \Pi_{0} \right) \mathrm{d}\mu_{F} \\ 
& = & \Lambda_{F} \int \limits_{\Omega_{{\rm DS}}} g \mathrm{d}\nu_{f}. 
\end{eqnarray*} 

\subsection{Proof of Theorem \ref{thm:ruelleuniqueconnnect}} 

We are now ready to prove Theorem \ref{thm:ruelleuniqueconnnect}. 

\begin{proof}[of Theorem \ref{thm:ruelleuniqueconnnect}]
Let $f \in \mathcal{C} \left( \Omega_{{\rm DS}}, \mathbb{R} \right)$ be a H\"{o}lder continuous function satisfying $\mathcal{L}_{f} \mathbb{1}_{\Omega_{{\rm DS}}} = e^{{\rm Pr} (\Gamma, f)} \mathbb{1}_{\Omega_{{\rm DS}}}$ and let $F = f \circ \Pi_{0}$. First note that, by Theorem \ref{thm:measureexistsruelle}, we have a measure $m_{F}$ satisfying
\begin{equation}
\label{eqn:ruellestar}
L_{F}^{*} m_{F}\ \ =\ \ e^{{\rm Pr} \left( \sigma^{\Gamma}, F \right)} m_{F}. 
\end{equation}

Thus, 
\[ L_{F} \left( \mathbb{1}_{\mathscr{P}^{\Gamma}_{{\rm inv}} \left( \Omega_{{\rm DS}} \right)} \right)\ \ =\ \ L_{\left( f \circ \Pi_{0} \right)} \left( \mathbb{1}_{\Omega_{{\rm DS}}} \circ \Pi_{0} \right)\ \ =\ \ \left( \mathcal{L}_{f} \mathbb{1}_{\Omega_{{\rm DS}}} \right) \circ \Pi_{0}\ \ =\ \ \left(e^{{\rm Pr} (\Gamma, f)} \mathbb{1}_{\Omega_{{\rm DS}}} \right) \circ \Pi_{0}, \] 
and hence we have, $L_{F} \mathbb{1}_{\mathscr{P}^{\Gamma}_{{\rm inv}} \left( \Omega_{{\rm DS}} \right)} = e^{{\rm Pr} \left( \sigma^{\Gamma}, F \right)} \mathbb{1}_{\mathscr{P}^{\Gamma}_{{\rm inv}} \left( \Omega_{{\rm DS}} \right)}$. This results in 
\begin{equation}
\label{eqn:ruelleoneds}
\widehat{L_{F}} \mathbb{1}_{\mathscr{P}^{\Gamma}_{{\rm inv}} \left( \Omega_{{\rm DS}} \right)}\ \ =\ \ \mathbb{1}_{\mathscr{P}^{\Gamma}_{{\rm inv}} \left( \Omega_{{\rm DS}} \right)}.
\end{equation} 

Equations \eqref{eqn:ruellestar} and \eqref{eqn:ruelleoneds} when combined with the conclusion from Theorem \ref{thm:mapseigenfn} (which also ensures $m_{F}$ is $\sigma^{\Gamma}$-invariant in this case) and the ideas from the proof of Theorem \ref{thm:uniqueness}, results in the fact that $\nu_{f} = \left( \Pi_{0} \right)_{*} m_{F}$ is the unique equilibrium state for $f$. The discussion at the end of Subsection \ref{subsec}, when applied to the Equation \eqref{eqn:ruellestar} implies that $\nu_{f}$ also satisfies $\mathcal{L}_{f}^{*} \nu_{f} = e^{{\rm Pr} (\Gamma, f)} \nu_{f}$, thereby completing the proof. 
\end{proof}

\begin{remark}
As we mentioned earlier, one can study the main theorems of this manuscript, namely Theorems \ref{thm:uniqueness} and \ref{thm:ruelleuniqueconnnect} with the domain of the correspondence being $\widehat{\mathbb{C}}$ provided we have topological transitivity of the map $\sigma^{\Gamma}$ on the space $\mathscr{P}^{\Gamma} \left( \widehat{\mathbb{C}} \right)$. The main reason for restricting the dynamics of $\Gamma$ to $\Omega_{{\rm DS}}$ and that of $\sigma^{\Gamma}$ to $\mathscr{P}^{\Gamma}_{{\rm inv}} \left( \Omega_{{\rm DS}} \right)$ is to use the density of pre-images of points in $\Omega_{{\rm DS}}$, which leads to transitivity. Further note that \cite{ss:2025} predominantly studied the Ruelle operator associated to the correspondence with the domain being $\widehat{\mathbb{C}}$. 
\end{remark}

\section{Comments on the hypothesis of Theorem \ref{thm:ruelleuniqueconnnect}} 
\label{sec:six} 

We sign off this article with a few comments on the hypothesis of Theorem \ref{thm:ruelleuniqueconnnect}. We also mention two examples of holomorphic correspondences that are related to the contents of this work.

Firstly, we mention about an obstruction regarding the normalisation of a real valued H\"{o}lder continuous function defined on $\Omega_{{\rm DS}}$. In the case of dynamics of maps, this is typically done, using a coboundary condition. Note that, as seen in Equation \eqref{eqn:normalmaps}, normalising a function $\phi \in \mathcal{C} \left( X, \mathbb{R} \right)$ involves the Koopman operator $U_{T} (h_{\phi}) = h_{\phi} \circ T$ associated to the function $h_{\phi}$ as obtained from Theorem \ref{thm:mapseigenfn}. While the work \cite{ss:2025} identified the eigenfunction $h_{f}$ associated to a potential $f \in \mathcal{C} \left( \Omega_{{\rm DS}}, \mathbb{R} \right)$ for the dynamics of a specific family of correspondences, the authors did not use the coboundary type condition to achieve their goal. A major obstruction to generalising the normalisation procedure to the case of correspondences is the fact that each point possibly has more than one image. Thus, the Koopman operator for correspondences looks a bit different. This object was considered in the works \cite{lon:2024} and \cite{sb:2026}. 

As defined in \cite{sb:2026}, for a holomorphic correspondence $\digamma$ with topological degree $d_{{\rm top}}$, the Koopman operator associated to a potential $f \in \mathcal{C} \left( \Omega_{{\rm DS}}, \mathbb{R} \right)$ is given by 
\[ \left( U_{\digamma} (f) \right) (z)\ \ =\ \ \frac{1}{d_{{\rm top}}} \sideset{}{'} \sum_{w\, \in\, \digamma^{\dagger}(z)} f(w). \] 

Now, consider $f \in \mathcal{C} \left( \Omega_{{\rm DS}}, \mathbb{R} \right)$. Suppose $h_{f}$ satisfies $\mathcal{L}_{f} h_{f} = e^{{\rm Pr} (\Gamma, f)} h_{f}$ (see Theorem 7.2 in \cite{ss:2025} for more details on this), then one may attempt to give a coboundary type condition for normalising $f$ by defining 
\begin{equation}
\label{eqn:normalcorr}
\hat{f}\ \ :=\ \ f + \log h_{f} - \log \left( U_{\digamma} (h_{f}) \right). 
\end{equation}
In the case of dynamics of maps, Equation \eqref{eqn:normalmaps} results in $\displaystyle{\widehat{L_{\psi}} = M_{\phi}^{-1} \circ \widehat{L_{\phi}} \circ M_{\phi}}$, for a suitable multiplication operator $M_{\phi}$, where $\psi = \hat{\phi}$. This step makes a lot of calculations easier. Such a nice characterisation has not yet been attained, in the case of correspondences and it is easy to observe that Equation \eqref{eqn:normalcorr} doesn't result in such a conclusion. So, one requires new ideas for doing this procedure in the setting of correspondences. A key property of normalisation is getting $\mathbb{1}_{\Omega_{{\rm DS}}}$ as the eigenfunction of the operator $\mathcal{L}_{f}$. This is the essential component in our investigation regarding the conclusion in Theorem \ref{thm:ruelleuniqueconnnect} as well and hence we made the assumption that $f \in \mathcal{C} \left( \Omega_{{\rm DS}}, \mathbb{R} \right)$ satisfies $\mathcal{L}_{f} \mathbb{1}_{\Omega_{{\rm DS}}} = e^{{\rm Pr} (\Gamma, f)} \mathbb{1}_{\Omega_{{\rm DS}}}$ in our hypothesis. Having explained the motivation behind the hypothesis, we now turn our attention to give examples of situations where this expression holds. 

Note that for a holomorphic correspondence $\Gamma$ with topological degree $d_{{\rm top}}$, if we choose the potential to be $f \equiv 0$, the constant zero function, we get 
\[ \left( \mathcal{L}_{f} \mathbb{1}_{\Omega_{{\rm DS}}} \right) (z)\ \ =\ \ \sum_{w\, \in\, \digamma^{\dagger} (z)} e^{f(w)} \mathbb{1}_{\Omega_{{\rm DS}}} (w)\ \ =\ \ d_{{\rm top}} \mathbb{1}_{\Omega_{{\rm DS}}} (z). \] 
Thus, it is sufficient for us to focus on examples of holomorphic correspondences $\Gamma$ satisfying ${\rm Pr} (\Gamma, 0) = \log \left( d_{{\rm top}} \right)$, meaning holomorphic correspondences on $\widehat{\mathbb{C}}$ whose topological entropy is equal to $\log \left( d_{{\rm top}} \right)$. 

Let $R_{1}, R_{2}, \cdots R_{n}$ be rational maps on $\widehat{\mathbb{C}}$ with their respective degrees given by $d_{1}, d_{2}, \cdots d_{n}$, with at least one of the maps having degree strictly greater than $1$. Consider the rational semigroup $S$ which is generated by these $n$ rational maps. It is well known that $S$ can be viewed as a holomorphic correspondence, which we shall denote by $\Gamma_{S}$. A key result from \cite{bs:2021} (see Corollary 1.9 in the paper) is that the topological entropy of $\Gamma_{S}$ is given by $\displaystyle{\log \left( \sum_{j\, =\, 1}^{n} d_{j} \right)}$. It also turns out that $\displaystyle{d_{S} := \sum_{j\, =\, 1}^{n} d_{j}}$ is the topological degree of $\Gamma_{S}$. If one wishes to work with the domain being $\Omega_{{\rm DS}}$ for the correspondence $\Gamma_{S}$, we need to ensure that the topological entropy remains to be $\log d_{S}$ even when $\Gamma_{S}$ is restricted to $\Omega_{{\rm DS}}$. Otherwise, one can proceed with the arrangement discussed earlier in this section. 

Note that for a finitely generated rational semigroup $S$ whose associated holomorphic correspondence is denoted by $\Gamma_{S}$, we have the set $\Omega_{{\rm DS}}$ to be equal to the Julia set $J_{S}$ of the rational semigroup $S$. The work \cite{hm:1996} studied these objects in a great detail. Also mentioned in this work is the concept of nearly Abelian rational semigroup. Suppose all the generators of $S$ have degree greater than or equal to $2$. Further assume that $S$ is nearly Abelian. Then, from Theorem 4.1 in \cite{hm:1996}, we have $J_{S} = J(g)$ for all maps $g \in S$. Here $J(g)$ denotes the Julia set of the rational map $g$. Thus, $J_{S}$ becomes completely invariant. Now, it follows from Theorem 4.2 of \cite{bs:2021} that the topological entropy of $\Gamma_{S}$ restricted to the domain $\Omega_{{\rm DS}} = J_{S}$ is $\log d_{S}$, where $d_{S}$ is the topological degree of the correspondence $\Gamma_{S}$. 

We end this paper by giving another family of examples of holomorphic correspondences on $\widehat{\mathbb{C}}$ that one can work with in this case. This example was first considered in \cite{bp:1994} and was investigated in detail in \cite{bl:2020}. For $a \in \widehat{\mathbb{C}} \setminus \{1\}$, define the correspondence $\mathcal{F}_{a}$ to be,
\[ \left( \dfrac{az + 1}{z + 1} \right)^{2} + \left( \dfrac{az + 1}{z + 1} \right) \left( \dfrac{aw - 1}{w - 1} \right) + \left( \dfrac{aw - 1}{w - 1} \right)^{2}\ \ =\ \ 3. \] 
It is interesting to study the correspondence $\mathcal{F}_{a}$ when the parameter $a$ belongs to a set called the \emph{Klein combination locus}, denoted by $\mathcal{K}$. This was studied in detail in \cite{bl:2020}. Also, see Definition 3.1 in the recent work \cite{vd:2026} for more details on this. Further, this work also computed the topological entropy of the family $\mathcal{F}_{a}$, where $a \in \mathcal{K} \subset \widehat{\mathbb{C}}$ to be $\log 2$. Note that the topological degree of $\mathcal{F}_{a}$ is $2$. Interested readers may refer to the proof of Theorem 0.2 in \cite{vd:2026} for details. 

\bigskip 

\emph{Authors' contact coordinates:} \\ 

{\bf Bharath Krishna Seshadri} \\ 
Indian Institute of Science Education and Research Thiruvananthapuram (IISER-TVM). \\ 
email: \texttt{bharathmaths21@iisertvm.ac.in}  
\medskip 

{\bf Shrihari Sridharan} \\ 
Indian Institute of Science Education and Research Thiruvananthapuram (IISER-TVM). \\ 
email: \texttt{shrihari@iisertvm.ac.in}

\end{document}